\documentclass[preprint,10pt]{elsarticle}
\usepackage{amssymb}
\usepackage{pifont}
\usepackage{amsfonts}
\usepackage{amsmath}
\usepackage{graphicx}
\usepackage{latexsym,amsthm}
\usepackage{hyperref}
\usepackage{combelow}
\usepackage{setspace}
\usepackage[left=1.3cm,top=1.75cm,right=1.3cm]{geometry}
\makeatletter
\def\ps@pprintTitle{%
  \let\@oddhead\@empty
  \let\@evenhead\@empty
  \def\@oddfoot{\reset@font\hfil\thepage\hfil}
  \let\@evenfoot\@oddfoot
}
\makeatother
\allowdisplaybreaks

\newtheorem{lemma}{Lemma}

\newtheorem{theorem}{Theorem}
\numberwithin{equation}{section}

\begin{document}

\begin{frontmatter}
\title{On Chlodowsky variant of $(p,q)$ Kantorovich-Stancu-Schurer operators}
\author[label1,label2]{Vishnu Narayan Mishra}
\ead{vishnunarayanmishra@gmail.com,
vishnu\_narayanmishra@yahoo.co.in}
\address[label1]{Department of Applied Mathematics \& Humanities,
Sardar Vallabhbhai National Institute of Technology, Ichchhanath Mahadev Dumas Road, Surat -395 007 (Gujarat), India}
\address[label2]{L. 1627 Awadh Puri Colony Beniganj,
Phase -III, Opp. Industrial Training Institute, Ayodhya Main Road, Faizabad-224 001,
(Uttar Pradesh), India } 
\author[label1, label*]{Shikha Pandey}
\ead{sp1486@gmail.com}
\fntext[label*]{Corresponding author}
\begin{abstract}
In the present paper, we introduce the Chlodowsky variant of $(p,q)$ Kantorovich-Stancu-Schurer operators on the unbounded domain which is a generalization of $(p,q)$ Bernstein-Stancu-Kantorovich operators. We have also derived its Korovkin type approximation properties and rate of convergence.
\end{abstract}
\end{frontmatter}
\section{Introduction and preliminaries}
Approximation theory has an important role in mathematical research because of its great potential for applications. Korovkin gave his famous approximation theorem in 1950, since then  the study of the linear methods of approximation given by sequences of positive and linear operators became a deep-rooted part of approximation theory. Considering it, various operators as Bernstein, Sz\'{a}sz, Baskakov etc. and their generalizations are being studied. In recent years, many results about the generalization of positive linear operators have been obtained by several mathematicians (\cite{Gairola},\cite{mishra1}-\cite{mishra6},\cite{BaskontoStancu}). In last two decades, the applications of $q$-calculus has played an important role in the area of approximation theory, number theory and theoretical physics. In 1987, Lupa\cb s  and in 1997, Phillips introduced a sequence of Bernstein polynomials based on $q$-integers and investigated its approximation properties. Several researchers obtained various other generalizations of operators based on $q$-calculus(See \cite{HM4},\cite{HM1}-\cite{HM3}).\\
Recently, Mursaleen et al. applied $(p, q)$-calculus in approximation theory
and introduced first $(p, q)$-analogue of Bernstein operators. They investigated uniform convergence of the operators and order of convergence, obtained Voronovskaja
theorem as well. Also, $(p, q)$-analogue of Bernstein operators, Bernstein-Stancu operators and Bernstein-Schurer operators, Kontorovich Bernstein Schurer, and Bleimann-Butzer-Hahn operators were introduced in (\cite{bernpq}-\cite{blehahn}),
respectively. Further, T. Acar \cite{acar} have studied recently, $(p, q)$-Generalization of Sz\'{a}sz–Mirakyan operators.\\
In the present paper, we introduce the Chlodowsky variant of $(p,q)$ Kantorovich-Stancu-Schurer operators on the unbounded domain.
We begin by recalling certain notations of $(p, q)$-calculus.\\
For $0<q<p \leq 1$, the $(p,q)$ integer $[n]_{p,q}$ is defined by
\begin{equation*}
[n]_{p,q} := \dfrac{p^n-q^n}{p-q}. 
\end{equation*}
$(p,q)$ factorial is expressed as
\begin{equation*}
[n]_{p,q}!=[n]_{p,q} [n-1]_{p,q} [n-2]_{p,q} \ldots 1.
\end{equation*}
$(p,q)$ binomial coefficient is expressed as
\begin{equation*}
\left[\begin{array}{c} n \\ k \end{array} \right]_{p,q} := \dfrac{[n]_{p,q}!}{[k]_{p,q}![n-k]_{p,q}!}.
\end{equation*}
$(p,q)$ binomial expansion as
\begin{equation*}
(ax+by)_{p,q}^n:=\sum_{k=0}^n\left[\begin{array}{c} n \\ k \end{array} \right]_{p,q}a^{n-k}b^k x^{n-k}y^k.
\end{equation*}
$$(x+y)_{p,q}^n := (x+y) (px+qy) (p^2x+q^2y)\ldots (p^{n-1}x+q^{n-1}y).$$
The definite integrals of the function $f$ are given by
$$\int_0^a f(x) d_{p,q}x =(q-p)a \sum_{k=0}^\infty \dfrac{p^k}{q^{k+1}}f\left(\dfrac{p^k}{q^{k+1}}a\right), ~~~~\left|\dfrac{p}{q}\right|<1,$$
and
$$\int_0^a f(x) d_{p,q}x= (p-q)a \sum_{k=0}^\infty \dfrac{q^k}{p^{k+1}}f\left(\dfrac{q^k}{p^{k+1}}a\right), ~~~~\left|\dfrac{p}{q}\right|>1.$$
Further $(p,q)$ analysis can be found in \cite{pqdiff}.\\
In 1932, Chlodowsky \cite{Ibikli} presented a generalization of Bernstein polynomials on an unbounded set, known as Bernstein-Chlodowsky polynomials given by,
\begin{equation}
B_n(f,x)=\sum_{k=0}^n f(\dfrac{k}{n}b_n)\left(\begin{array}{c} n \\ k \end{array} \right) \left(\dfrac{x}{b_n}\right)^k \left(1-\dfrac{x}{b_n}\right)^{n-k}, ~~~~~ 0\leq x\leq b_n,
\end{equation}
where $b_n$ is  an increasing sequence of positive terms with the
properties $b_n \rightarrow \infty$ and $\frac{b_n}{n} \rightarrow 0$ as $n \rightarrow \infty.$\vspace{.5 cm}\\
In 2015, Vedi and \"{O}zarslan \cite{bskq} investigated Chlodowsky-type
$q$-Bernstein-Stancu-Kantorovich operators,
and Wafi and Rao investigated $(p,q)$ form of Kantorovich type Bernstein-Stancu-Schurer operator. Mursaleen and Khan \cite{bernpqkonto} defined the Kantorovich type $(p, q)$-Bernstein-Schurer Operators, given by
\begin{eqnarray*}
&& T_{n,m}(f;x,p,q)=\sum_{k=0}^{n+m} \left[\begin{array}{c} n+m \\ k \end{array} \right]_{p,q} x^k \prod_{s=0}^{n+m-k-1} (p^s-q^sx) \int_0^1 f \left(\dfrac{(1-t)[k]_{p,q}+[k+1]_{p,q}t}{[n+1]_{p,q}}\right)\mathop{d_{p,q}t} \\&& k=0,1,2,\ldots,n=1,2,3,\ldots
\end{eqnarray*}
where 
\begin{lemma} \label{L1}(See \cite{bernpqkonto}) For the Operators $T_{n,m}^{(\alpha,\beta)},$  we have
\begin{eqnarray*}
 && T_{n,m}(1;x,p,q)=1,\\&&
 T_{n,m}(t;x,p,q)=\dfrac{(px+1-x)_{p,q}^{n+m}}{[2]_{p,q}[n+1]_{p,q}}+\dfrac{(p+2q-1)[n+m]_{p,q}}{[2]_{p,q}[n+1]_{p,q}}x,\\&&
T_{n,m}(t^2;x,p,q)=\dfrac{(p^2x+1-x)_{p,q}^{n+m}}{[3]_{p,q}[n+1]_{p,q}^2}+\left\{ 1+\dfrac{2q}{[2]_{p,q}}+\dfrac{q^2-1}{[3]_{p,q}}\right\} \dfrac{[n+m]_{p,q}}{[n+1]_{p,q}^2}(px+1-x)_{p,q}^{n+m-1}x\\&&~~~~~~~~~~~~~~~~~~~~~~+\Biggl\{ 1+\dfrac{2(q-1)}{[2]_{p,q}}+\dfrac{(q-1)^2}{[3]_{p,q}}\Biggl\}\dfrac{[n+m]_{p,q}[n+m-1]_{p,q}}{[n+1]_{p,q}^2}x^2.
\end{eqnarray*}
\end{lemma}
\section{Construction of the operators}
We construct the Chlodowsky variant of $(p,q)$ Kontorovich-Stancu-Schurer operators as
\begin{equation} \label{1}
K_{n,m}^{(\alpha,\beta)}(f;x,p,q)=\sum_{k=0}^{n+m} \left[\begin{array}{c} n+m \\ k \end{array} \right]_{p,q}\prod_{s=0}^{n+m-k-1} (p^s-q^s\dfrac{x}{b_n}) \left(\dfrac{x}{b_n}\right)^k  \int_0^1 f \left(\dfrac{(1-t)[k]_{p,q}+[k+1]_{p,q}t+\alpha}{[n+1]_{p,q}+\beta}b_n\right)\mathop{d_{p,q}t},
\end{equation}
where $n\in\mathbb{N}$, $m,\alpha,\beta \in\mathbb{N}_0$ with $0 \leq \alpha \leq \beta$, $0 \leq x \leq b_n$, $0<q< p\leq1.$\\ Obviously, $K_{n,m}^{(\alpha,\beta)}$
is a linear and positive operator. 
To begin with, we obtain the following important lemma.
\begin{lemma}\label{L2} Let $K_{n,m}^{(\alpha,\beta)}(f;x,p,q)$ be given by (\ref{1}). The first few moments of the operators are
\begin{eqnarray*}
&&(i)~ K_{n,m}^{(\alpha,\beta)}(1;x,p,q)=1,\\
&& (ii)~ K_{n,m}^{(\alpha,\beta)}(t;x,p,q)=\Biggl(\dfrac{1}{[n+1]_{p,q}+\beta}\Biggl) \Biggl(\alpha b_n+\dfrac{(p\frac{x}{b_n}+1-\frac{x}{b_n})_{p,q}^{n+m}}{[2]_{p,q}}b_n+\dfrac{(p+2q-1)[n+m]_{p,q}}{[2]_{p,q}}x\Biggl),\\
&&(iii)~ K_{n,m}^{(\alpha,\beta)}(t^2;x,p,q)=\Biggl(\dfrac{1}{[n+1]_{p,q}+\beta}\Biggl)^2 \Biggl[\left(\alpha^2+\dfrac{2\alpha}{[2]_{p,q}}\left(p\frac{x}{b_n}+1-\frac{x}{b_n}\right)_{p,q}^{n+m}+\dfrac{(p^2\frac{x}{b_n}+1-\frac{x}{b_n})_{p,q}^{n+m}}{[3]_{p,q}}\right)b_n^2
\\&&\hspace{4 cm}+\left(\dfrac{2\alpha}{[2]_{p,q}}(p+2q-1)+\left\{ 1+\dfrac{2q}{[2]_{p,q}}+\dfrac{q^2-1}{[3]_{p,q}}\right\}  \left(p\frac{x}{b_n}+1-\frac{x}{b_n}\right)_{p,q}^{n+m-1}\right)[n+m]_{p,q}b_nx
\\&&\hspace{4 cm}+\Biggl\{ 1+\dfrac{2(q-1)}{[2]_{p,q}}+\dfrac{(q-1)^2}{[3]_{p,q}}\Biggl\}[n+m]_{p,q}[n+m-1]_{p,q}x^2\Biggl],\\
&& (iv)~ K_{n,m}^{(\alpha,\beta)}((t-x);x,p,q)= \dfrac{[2]_{p,q}\alpha+(p\frac{x}{b_n}+1-\frac{x}{b_n})_{p,q}^{n+m}}{[2]_{p,q}([n+1]_{p,q}+\beta)}b_n+\Biggl(\dfrac{(p+2q-1)[n+m]_{p,q}}{[2]_{p,q}([n+1]_{p,q}+\beta)}-1\Biggl)x
\\
&& (v)~ K_{n,m}^{(\alpha,\beta)}((t-x)^2;x,p,q)= \left[\dfrac{\alpha^2}{([n+1]_{p,q}+\beta)^2}+\dfrac{2\alpha}{[2]_{p,q}([n+1]_{p,q}+\beta)^2}\left(p\frac{x}{b_n}+1-\frac{x}{b_n}\right)_{p,q}^{n+m}+\dfrac{(p^2\frac{x}{b_n}+1-\frac{x}{b_n})_{p,q}^{n+m}}{[3]_{p,q}([n+1]_{p,q}+\beta)^2}\right]b_n^2 \\&&\hspace{4.5 cm}
+\Biggl[\dfrac{2\alpha(p+2q-1)[n+m]_{p,q}}{[2]_{p,q}([n+1]_{p,q}+\beta)^2}+
\left\{ 1+\dfrac{2q}{[2]_{p,q}}+\dfrac{q^2-1}{[3]_{p,q}}\right\} \dfrac{[n+m]_{p,q}}{([n+1]_{p,q}+\beta)^2} \left(p\frac{x}{b_n}+1-\frac{x}{b_n}\right)_{p,q}^{n+m-1}
\\&& \hspace{5.5 cm}-\dfrac{2\alpha}{([n+1]_{p,q}+\beta)} - \dfrac{2(p\frac{x}{b_n}+1-\frac{x}{b_n})_{p,q}^{n+m}}{[2]_{p,q}([n+1]_{p,q}+\beta)}
\Biggl]b_nx\\&& \hspace{4.5 cm}+\Biggl[\Biggl\{ 1+\dfrac{2(q-1)}{[2]_{p,q}}+\dfrac{(q-1)^2}{[3]_{p,q}}\Biggl\}\dfrac{[n+m]_{p,q}[n+m-1]_{p,q}}{([n+1]_{p,q}+\beta)^2}-2 \dfrac{(p+2q-1)[n+m]_{p,q}}{[2]_{p,q}([n+1]_{p,q}+\beta)} +1
 \Biggl]x^2.
\end{eqnarray*} 
\end{lemma}
\begin{proof}
From operator \ref{1},
\begin{eqnarray} \label{2}
&&K_{n,m}^{(\alpha,\beta)}(t^u;x,p,q)= \sum_{k=0}^{n+m} \left[\begin{array}{c} n+m \\ k \end{array} \right]_{p,q}\prod_{s=0}^{n+m-k-1} (p^s-q^s\dfrac{x}{b_n}) \left(\dfrac{x}{b_n}\right)^k  \int_0^1  \left(\dfrac{(1-t)[k]_{p,q}+[k+1]_{p,q}t+\alpha}{[n+1]_{p,q}+\beta}b_n\right)^u\mathop{d_{p,q}t}\nonumber \\&&\qquad \qquad \qquad ~~~=
\sum_{k=0}^{n+m} \left[\begin{array}{c} n+m \\ k \end{array} \right]_{p,q}\prod_{s=0}^{n+m-k-1} (p^s-q^s\dfrac{x}{b_n}) \left(\dfrac{x}{b_n}\right)^k \dfrac{[n+1]_{p,q}^u b_n^u}{([n+1]_{p,q}+\beta)^u} \nonumber \\&&\qquad \qquad \qquad ~~~~~~ \times \int_0^1  \left(\dfrac{(1-t)[k]_{p,q}+[k+1]_{p,q}t+\alpha}{[n+1]_{p,q}}\right)^u\mathop{d_{p,q}t}\nonumber \\&&\qquad \qquad \qquad ~~~=
\dfrac{[n+1]_{p,q}^u b_n^u}{([n+1]_{p,q}+\beta)^u} \sum_{k=0}^{n+m} \left[\begin{array}{c} n+m \\ k \end{array} \right]_{p,q}\left(\dfrac{x}{b_n}\right)^k \prod_{s=0}^{n+m-k-1} (p^s-q^s\dfrac{x}{b_n})  \sum_{i=0}^u \left(\begin{array}{c} u \\ i \end{array} \right)\Biggl(\dfrac{\alpha}{[n+1]_{p,q}} \Biggl)^{u-i}\nonumber \\&& \qquad \qquad \qquad ~~~~~~\times \int_0^1  \left(\dfrac{(1-t)[k]_{p,q}+[k+1]_{p,q}t}{[n+1]_{p,q}}\right)^u\mathop{d_{p,q}t}\nonumber \\&&\qquad \qquad \qquad ~~~
=\dfrac{[n+1]_{p,q}^u b_n^u}{([n+1]_{p,q}+\beta)^u} \sum_{i=0}^u \left(\begin{array}{c} u \\ i \end{array} \right)\Biggl(\dfrac{\alpha}{[n+1]_{p,q}} \Biggl)^{u-i}~~ \sum_{k=0}^{n+m} \left[\begin{array}{c} n+m \\ k \end{array} \right]_{p,q}\left(\dfrac{x}{b_n}\right)^k \prod_{s=0}^{n+m-k-1} (p^s-q^s\dfrac{x}{b_n})  \nonumber \\&& \qquad \qquad \qquad ~~~~~~ \times \int_0^1  \left(\dfrac{(1-t)[k]_{p,q}+[k+1]_{p,q}t}{[n+1]_{p,q}}\right)^i\mathop{d_{p,q}t}\nonumber \\&&
K_{n,m}^{(\alpha,\beta)}(t^u;x,p,q)=\dfrac{[n+1]_{p,q}^u b_n^u}{([n+1]_{p,q}+\beta)^u} \sum_{i=0}^u \left(\begin{array}{c} u \\ i \end{array} \right)\Biggl(\dfrac{\alpha}{[n+1]_{p,q}} \Biggl)^{u-i} T_{n,m}(t^i;\frac{x}{b_n},p,q).
\end{eqnarray}
Thus for u=0,1,2 we get
\begin{eqnarray*}
&&K_{n,m}^{(\alpha,\beta)}(1;x,p,q)=T_{n,m}(1;\frac{x}{b_n},p,q),\\
&&K_{n,m}^{(\alpha,\beta)}(t;x,p,q)=\dfrac{[n+1]_{p,q} b_n}{[n+1]_{p,q}+\beta} \sum_{i=0}^1 \left(\begin{array}{c} 1 \\ i \end{array} \right)\Biggl(\dfrac{\alpha}{[n+1]_{p,q}} \Biggl)^{1-i} T_{n,m}(t^i;\frac{x}{b_n},p,q)\\&&\qquad \qquad \qquad ~~=\Biggl(\dfrac{[n+1]_{p,q}}{[n+1]_{p,q}+\beta}\Biggl)b_n\left\{\dfrac{\alpha}{[n+1]_{p,q}}+T_{n,m}(t;\frac{x}{b_n},p,q)\right\}, \\
&&K_{n,m}^{(\alpha,\beta)}(t^2;x,p,q)=\dfrac{[n+1]_{p,q}^2 b_n^2}{([n+1]_{p,q}+\beta)^2} \sum_{i=0}^2 \left(\begin{array}{c} 2 \\ i \end{array} \right)\Biggl(\dfrac{\alpha}{[n+1]_{p,q}} \Biggl)^{2-i} T_{n,m}(t^i;\frac{x}{b_n},p,q)\\&&\qquad \qquad \qquad ~~=\dfrac{[n+1]_{p,q}^2 b_n^2}{([n+1]_{p,q}+\beta)^2} \Biggl[ \Biggl(\dfrac{\alpha}{[n+1]_{p,q}} \Biggl)^2+2\Biggl(\dfrac{\alpha}{[n+1]_{p,q}} \Biggl)T_{n,m}(t;\frac{x}{b_n},p,q)  +T_{n,m}(t^2;\frac{x}{b_n},p,q)\Biggl].
\end{eqnarray*}
Using Lemma \ref{L1} and in view of the above relations we get the statements (i), (ii) and (iii).\\
Using linear property of operators, we have
\begin{eqnarray*}
&&K_{n,m}^{(\alpha,\beta)}((t-x);x,p,q)=K_{n,m}^{(\alpha,\beta)}(t;x,p,q)-xK_{n,m}^{(\alpha,\beta)}(1;x,p,q)\\
&&\hspace{3.5 cm}= \dfrac{[2]_{p,q}\alpha+(p\frac{x}{b_n}+1-\frac{x}{b_n})_{p,q}^{n+m}}{[2]_{p,q}([n+1]_{p,q}+\beta)}b_n+\Biggl(\dfrac{(p+2q-1)[n+m]_{p,q}}{[2]_{p,q}([n+1]_{p,q}+\beta)}-1\Biggl)x.
\end{eqnarray*}
Hence, we get (iv).\\
Similar calculations give
$$
K_{n,m}^{(\alpha,\beta)}((t-x)^2;x,p,q)=K_{n,m}^{(\alpha,\beta)}(t^2;x,p,q)-2xK_{n,m}^{(\alpha,\beta)}(t;x,p,q)+x^2K_{n,m}^{(\alpha,\beta)}(1;x,p,q).$$
Then we obtain,
\begin{eqnarray*}
&&K_{n,m}^{(\alpha,\beta)}((t-x)^2;x,p,q)= \left[\dfrac{\alpha^2}{([n+1]_{p,q}+\beta)^2}+\dfrac{2\alpha}{[2]_{p,q}([n+1]_{p,q}+\beta)^2}\left(p\frac{x}{b_n}+1-\frac{x}{b_n}\right)_{p,q}^{n+m}+\dfrac{(p^2\frac{x}{b_n}+1-\frac{x}{b_n})_{p,q}^{n+m}}{[3]_{p,q}([n+1]_{p,q}+\beta)^2}\right]b_n^2 \\&&\hspace{4 cm}
+\Biggl[\dfrac{2\alpha(p+2q-1)[n+m]_{p,q}}{[2]_{p,q}([n+1]_{p,q}+\beta)^2}+
\left\{ 1+\dfrac{2q}{[2]_{p,q}}+\dfrac{q^2-1}{[3]_{p,q}}\right\} \dfrac{[n+m]_{p,q}}{([n+1]_{p,q}+\beta)^2} \left(p\frac{x}{b_n}+1-\frac{x}{b_n}\right)_{p,q}^{n+m-1}
\\&& \hspace{5 cm}-\dfrac{2\alpha}{([n+1]_{p,q}+\beta)} - \dfrac{2(p\frac{x}{b_n}+1-\frac{x}{b_n})_{p,q}^{n+m}}{[2]_{p,q}([n+1]_{p,q}+\beta)}
\Biggl]b_nx\\&& \hspace{4 cm}+\Biggl[\Biggl\{ 1+\dfrac{2(q-1)}{[2]_{p,q}}+\dfrac{(q-1)^2}{[3]_{p,q}}\Biggl\}\dfrac{[n+m]_{p,q}[n+m-1]_{p,q}}{([n+1]_{p,q}+\beta)^2}-2 \dfrac{(p+2q-1)[n+m]_{p,q}}{[2]_{p,q}([n+1]_{p,q}+\beta)} +1
 \Biggl]x^2.
\end{eqnarray*}
This proves (v).
\end{proof}

\section{Korovkin-type approximation theorem}
Assume $C_\rho$ is the space of all continuous functions $f$ such that $$|f(x)| \leq M \rho(x),~~~~~~~~~ -\infty<x<\infty.$$ 
Then $C_\rho$ is a Banach space with the norm $$\|f\|_\rho=\sup_{-\infty<x<\infty}\dfrac{|f(x)|}{\rho(x)}.$$
The subsequent results are used for proving Korovkin approximation theorem on unbounded sets.
\begin{theorem} \label{TT1}(See \cite{korovUnbounded}) There exists a sequence of positive linear operators $\mathbb{U}_n$, acting from $C_\rho$ to $C_\rho ,$ satisfying the conditions
\begin{enumerate}
\item[(1)] $\lim_{n\rightarrow \infty}\|\mathbb{U}_n(1;x)-1\|_\rho=0,$
\item[(2)]$\lim_{n\rightarrow \infty}\|\mathbb{U}_n(\phi;x)-\phi\|_\rho=0,$
\item[(3)]$\lim_{n\rightarrow \infty}\|\mathbb{U}_n(\phi^2;x)-\phi^2\|_\rho=0,$
\end{enumerate}
where $\phi(x)$ is a continuous and increasing function on $(-\infty,\infty)$ such that $\lim_{x\rightarrow \pm \infty} \phi(x) =
\pm \infty$ and $\rho(x) = 1 + \phi^2$, and there exists a function $f^* \in C_\rho$ for which $\varlimsup_{n \to \infty} \|\mathbb{U}_nf^* - f^*\|_\rho >0 $.\end{theorem}
\begin{theorem} \label{TT2} (See \cite{korovUnbounded})
Conditions (1),(2),(3) of above theorem implies that $$\lim_{n\to \infty}\|\mathbb{U}_nf-f\|_\rho=0$$ for any function $f$ belonging to the subset $C_\rho^0 := \{ f \in C_\rho : lim_{|x| \to\infty} \frac{|f(x)|}{\rho(x)} ~is~ finite\}$ .
\end{theorem}
Consider the weight function $\rho(x)=1+x^2$ and operators:
$$\mathbb{U}_{n,m}^{\alpha,\beta}(f;x,p,q)=\begin{cases}
K_{n,m}^{\alpha,\beta}(f;x,p,q)&\mathrm{if~} x \in [0,b_n],\\
f(x)& \mathrm{if~} x \in [0,\infty)/[0,b_n].\\
\end{cases}$$\\
 Thus for $f \in C_{1+x^2},$ we have
 \begin{align*}
 \|\mathbb{U}_{n,m}^{\alpha,\beta}(f;\cdot ,p,q)\| &\leq \sup_{x\in[0,b_n]} \dfrac{|\mathbb{U}_{n,m}^{\alpha,\beta}(f;x,p,q)|}{1+x^2}+\sup_{b_n<x<\infty}\dfrac{|f(x)|}{1+x^2} \\
 & \leq \|f\|_{1+x^2} \Biggl[ \sup_{x\in[0,\infty)} \dfrac{|\mathbb{U}_{n,m}^{\alpha,\beta}(1+t^2;x,p,q)|}{1+x^2}+1\Biggl].
 \end{align*}

Now, using Lemma \ref{L2}, we will obtain,
$$\|\mathbb{U}_{n,m}^{\alpha,\beta}(f;\cdot ,p,q)\|_{1+x^2} \leq M \|f\|_{1+x^2}$$
if $p:=(p)_n,q:=(q)_n$ with $0<q_n < p_n\leq 1,\lim_{n\to \infty}p_n=1, \lim_{n\to \infty}q_n=1,\lim_{n\to \infty}p_n^n=\lim_{n\to \infty}q_n^n=N, N<\infty ~and~ \lim_{n\to \infty}\frac{b_n}{[n]}=0.$
\begin{theorem} \label{TT3}
For all $f\in C_{1+x^2}^0$, we have $$\lim_{n\to \infty}\|\mathbb{U}_{n,m}^{\alpha,\beta}(f;\cdot ,p_n,q_n)-f(\cdot)\|_{1+x^2}=0$$
provided that $p:=(p)_n,q:=(q)_n$ with $0<q_n < p_n\leq 1,\lim_{n\to \infty}p_n=1, \lim_{n\to \infty}q_n=1,\lim_{n\to \infty}p_n^n=\lim_{n\to \infty}q_n^n=N, N<\infty ~and~ \lim_{n\to \infty}\frac{b_n}{[n]}=0.$
\end{theorem}
\begin{proof}
Using the results of Theorem \ref{TT1} and Lemma \ref{L2}(i),(ii) and (iii), we will achieve the following
assessments, respectively:
$$\sup_{x\in[0,\infty)}\dfrac{|\mathbb{U}_{n,m}^{\alpha,\beta}(1;x,p_n,q_n)-1|}{1+x^2}=\sup_{0\leq x \leq b_n}\dfrac{|K_{n,m}^{\alpha,\beta}(1;x,p_n,q_n)-1|}{1+x^2}=0$$
\begin{eqnarray*}
&&\sup_{x\in[0,\infty)}\dfrac{|\mathbb{U}_{n,m}^{\alpha,\beta}(t;x,p_n,q_n)-t|}{1+x^2}=\sup_{0\leq x \leq b_n}\dfrac{|K_{n,m}^{\alpha,\beta}(t;x,p_n,q_n)-x|}{1+x^2} \\
&&\hspace{1 cm} \leq 
 \sup_{0\leq x \leq b_n}\dfrac{\dfrac{[2]_{p,q}\alpha+(p\frac{x}{b_n}+1-\frac{x}{b_n})_{p,q}^{n+m}}{[2]_{p,q}([n+1]_{p,q}+\beta)}b_n+\Biggl(\dfrac{(p+2q-1)[n+m]_{p,q}}{[2]_{p,q}([n+1]_{p,q}+\beta)}-1\Biggl)x}{1+x^2} \\
&& \hspace{1 cm}\leq  \dfrac{[2]_{p,q}\alpha+(p\frac{x}{b_n}+1-\frac{x}{b_n})_{p,q}^{n+m}}{[2]_{p,q}([n+1]_{p,q}+\beta)}b_n+\Biggl|\dfrac{(p+2q-1)[n+m]_{p,q}}{[2]_{p,q}([n+1]_{p,q}+\beta)}-1\Biggl| \rightarrow 0
\end{eqnarray*}
and\\
\begin{eqnarray*}
&&\sup_{x\in[0,\infty)}\dfrac{|\mathbb{U}_{n,m}^{\alpha,\beta}(t^2;x,p_n,q_n)-t^2|}{1+x^2}=\sup_{ 0\leq x \leq b_n}\dfrac{|K_{n,m}^{\alpha,\beta}(t^2;x,p_n,q_n)-x^2|}{1+x^2}\\
&& \hspace{1 cm} \leq  \sup_{ 0\leq x \leq b_n}\dfrac{1}{1+x^2}\Biggl(\left[\dfrac{\alpha^2}{([n+1]_{p,q}+\beta)^2}+\dfrac{2\alpha}{[2]_{p,q}([n+1]_{p,q}+\beta)^2}\left(p\frac{x}{b_n}+1-\frac{x}{b_n}\right)_{p,q}^{n+m}+\dfrac{(p^2\frac{x}{b_n}+1-\frac{x}{b_n})_{p,q}^{n+m}}{[3]_{p,q}([n+1]_{p,q}+\beta)^2}\right]b_n^2 \\&&\hspace{2 cm}
+\Biggl[\dfrac{2\alpha(p+2q-1)[n+m]_{p,q}}{[2]_{p,q}([n+1]_{p,q}+\beta)^2}+
\left\{ 1+\dfrac{2q}{[2]_{p,q}}+\dfrac{q^2-1}{[3]_{p,q}}\right\} \dfrac{[n+m]_{p,q}}{([n+1]_{p,q}+\beta)^2} \left(p\frac{x}{b_n}+1-\frac{x}{b_n}\right)_{p,q}^{n+m-1}
\\&& \hspace{2.5 cm}-\dfrac{2\alpha}{([n+1]_{p,q}+\beta)} - \dfrac{2(p\frac{x}{b_n}+1-\frac{x}{b_n})_{p,q}^{n+m}}{[2]_{p,q}([n+1]_{p,q}+\beta)}
\Biggl]b_nx\\&& \hspace{2 cm}+\Biggl[\Biggl\{ 1+\dfrac{2(q-1)}{[2]_{p,q}}+\dfrac{(q-1)^2}{[3]_{p,q}}\Biggl\}\dfrac{[n+m]_{p,q}[n+m-1]_{p,q}}{([n+1]_{p,q}+\beta)^2}-2 \dfrac{(p+2q-1)[n+m]_{p,q}}{[2]_{p,q}([n+1]_{p,q}+\beta)} +1 \Biggl]x^2
\Biggl)\\
&& \hspace{1 cm} \leq \Biggl(\left[\dfrac{\alpha^2}{([n+1]_{p,q}+\beta)^2}+\dfrac{2\alpha}{[2]_{p,q}([n+1]_{p,q}+\beta)^2}\left(p\frac{x}{b_n}+1-\frac{x}{b_n}\right)_{p,q}^{n+m}+\dfrac{(p^2\frac{x}{b_n}+1-\frac{x}{b_n})_{p,q}^{n+m}}{[3]_{p,q}([n+1]_{p,q}+\beta)^2}\right]b_n^2 \\&&\hspace{2 cm}
+\Biggl|\dfrac{2\alpha(p+2q-1)[n+m]_{p,q}}{[2]_{p,q}([n+1]_{p,q}+\beta)^2}+
\left\{ 1+\dfrac{2q}{[2]_{p,q}}+\dfrac{q^2-1}{[3]_{p,q}}\right\} \dfrac{[n+m]_{p,q}}{([n+1]_{p,q}+\beta)^2} \left(p\frac{x}{b_n}+1-\frac{x}{b_n}\right)_{p,q}^{n+m-1}
\\&& \hspace{2.5 cm}-\dfrac{2\alpha}{([n+1]_{p,q}+\beta)} - \dfrac{2(p\frac{x}{b_n}+1-\frac{x}{b_n})_{p,q}^{n+m}}{[2]_{p,q}([n+1]_{p,q}+\beta)}
\Biggl|b_n\\&& \hspace{2 cm}+\Biggl|
\Biggl\{ 1+\dfrac{2(q-1)}{[2]_{p,q}}+\dfrac{(q-1)^2}{[3]_{p,q}}\Biggl\}\dfrac{[n+m]_{p,q}[n+m-1]_{p,q}}{([n+1]_{p,q}+\beta)^2}-2 \dfrac{(p+2q-1)[n+m]_{p,q}}{[2]_{p,q}([n+1]_{p,q}+\beta)} +1 
 \Biggl|~ \Biggl) \longrightarrow   0
\end{eqnarray*}
whenever $n \to \infty$, because we have $\lim_{n\to \infty}p_n=\lim_{n\to \infty}q_n =1$ and $\frac{b_n}{[n]_{p,q}}=0$ as $n \to \infty$.
\end{proof}
\begin{theorem}\label{TT4} Assuming $C$ as a positive and real number independent of $n$ and $f$ as a continuous
function which vanishes on $[C,\infty)$. Let $p := (p_n), q := (q_n)$ with $0 < q_n < p_n \leq 1, \lim_{n\to \infty} p_n =\lim_{n\to \infty} q_n= 1,~\lim_{n\to \infty} p_n^n = \lim_{n\to \infty} q_n^n= N < \infty $ and $\lim_{n\to \infty} \frac{b_n^2}{[n]_{p,q}} =0$. Then we have
 $$\lim_{n\to \infty}\sup_{0 \leq x\leq b_n} \left| K_{n,m}^{\alpha,\beta}(f;x,p_n,q_n)- f(x) \right|=0.$$ 
\end{theorem}
\begin{proof}
 From the hypothesis on $f$, it is bounded i.e. $|f(x)| \leq M ~(M>0)$. For any $\epsilon >0$, we have
$$\left|f \left(\dfrac{(1-t)[k]_{p,q}+[k+1]_{p,q}t+\alpha}{[n+1]_{p,q}+\beta}b_n\right)-f(x)\right|<\epsilon+\dfrac{2M}{\delta^2}\left(\dfrac{(1-t)[k]_{p,q}+[k+1]_{p,q}t+\alpha}{[n+1]_{p,q}+\beta}b_n-x\right)^2,$$
where $x\in[0,b_n]$ and $\delta=\delta(\epsilon)$
are independent of $n$. Now since we know,
$$K_{n,m}^{(\alpha,\beta)}((t-x)^2;x,p_n,q_n)=\sum_{k=0}^{n+m} \left[\begin{array}{c} n+m \\ k \end{array} \right]_{p,q}\prod_{s=0}^{n+m-k-1} (p^s-q^s\frac{x}{b_n}) \left(\frac{x}{b_n}\right)^k  \int_0^1 \left(\dfrac{(1-t)[k]_{p,q}+[k+1]_{p,q}t+\alpha}{[n+1]_{p,q}+\beta}b_n\right)^2\mathop{dt}.$$
we can conclude by Theorem \ref{TT3},
\begin{eqnarray*}
&&\sup_{0 \leq x\leq b_n} \left| K_{n,m}^{\alpha,\beta}(f;x,p_n,q_n)- f(x) \right| \leq \epsilon +\dfrac{2M}{\delta^2} b_n^2\Biggl(\Biggl[\dfrac{\alpha^2}{([n+1]_{p,q}+\beta)^2}+\dfrac{2\alpha}{[2]_{p,q}([n+1]_{p,q}+\beta)^2}\left(p\frac{x}{b_n}+1-\frac{x}{b_n}\right)_{p,q}^{n+m}\\&&\hspace{4 cm}+\dfrac{(p^2\frac{x}{b_n}+1-\frac{x}{b_n})_{p,q}^{n+m}}{[3]_{p,q}([n+1]_{p,q}+\beta)^2}\Biggl] \\&&\hspace{2 cm}
+\Biggl|\dfrac{2\alpha(p+2q-1)[n+m]_{p,q}}{[2]_{p,q}([n+1]_{p,q}+\beta)^2}+
\left\{ 1+\dfrac{2q}{[2]_{p,q}}+\dfrac{q^2-1}{[3]_{p,q}}\right\} \dfrac{[n+m]_{p,q}}{([n+1]_{p,q}+\beta)^2} \left(p\frac{x}{b_n}+1-\frac{x}{b_n}\right)_{p,q}^{n+m-1}
\\&& \hspace{2.5 cm}-\dfrac{2\alpha}{([n+1]_{p,q}+\beta)} - \dfrac{2(p\frac{x}{b_n}+1-\frac{x}{b_n})_{p,q}^{n+m}}{[2]_{p,q}([n+1]_{p,q}+\beta)}
\Biggl|\\&& \hspace{2 cm}+\Biggl|
\Biggl\{ 1+\dfrac{2(q-1)}{[2]_{p,q}}+\dfrac{(q-1)^2}{[3]_{p,q}}\Biggl\}\dfrac{[n+m]_{p,q}[n+m-1]_{p,q}}{([n+1]_{p,q}+\beta)^2}-2 \dfrac{(p+2q-1)[n+m]_{p,q}}{[2]_{p,q}([n+1]_{p,q}+\beta)} +1 
 \Biggl|~ \Biggl).
\end{eqnarray*}
Since $\frac{b_n^2}{[n]_{p,q}} =0$ as $n \to \infty,$ we have the desired result.\end{proof}
\section{Rate of Convergence}
We will find the rate of convergence in terms of the Lipschitz class $Lip_M(\gamma )$, for $0 < \gamma \leq 1$. Assume that $C_B[0,\infty)$ denote the space of bounded continuous functions on $[0,\infty)$. A function $f \in C_B[0,\infty)$ belongs to $Lip_M(\gamma )$ if
$$ |f (t) - f (x)| \leq M|t - x|^\gamma,~~~~ t, x \in [0,\infty)$$
is satisfied.\\
\begin{theorem} \label{T1} Let $f \in Lip_M(\gamma )$, then
$$K_{n,m}^{\alpha,\beta}(f;x,p,q)\leq M(\mu_{n,p,q}(x))^{\gamma /2},$$ where
$\mu_{n,p,q}(x)=K_{n,m}^{\alpha,\beta}((t-x)^2;x,p,q).$\end{theorem}
\begin{proof}
Since $f \in Lip_M(\gamma )$, and the operator $K_{n,m}^{\alpha,\beta}(f;x,p,q)$ is linear and monotone,
\begin{eqnarray*}
&&| K_{n,m}^{\alpha,\beta}(f;x,p,q)- f(x)| \\&& \hspace{1 cm}= \Biggl|\sum_{k=0}^{n+m} \left[\begin{array}{c} n+m \\ k \end{array} \right]_{p,q}  \prod_{s=0}^{n+m-k-1} (p^s-q^s\dfrac{x}{b_n}) \left(\dfrac{x}{b_n}\right)^k \\&& \hspace{2 cm} \times \int_0^1 \left( f \left(\dfrac{(1-t)[k]_{p,q}+[k+1]_{p,q}t+\alpha}{[n+1]_{p,q}+\beta}b_n\right)-f(x)\right)\mathop{d_{p,q}t}\Biggl|\\
&& \hspace{1 cm} \leq  \sum_{k=0}^{n+m} \left[\begin{array}{c} n+m \\ k \end{array} \right]_{p,q}  \prod_{s=0}^{n+m-k-1} (p^s-q^s\dfrac{x}{b_n})\left(\dfrac{x}{b_n}\right)^k \\&& \hspace{2 cm} \times \int_0^1 \left| f \left(\dfrac{(1-t)[k]_{p,q}+[k+1]_{p,q}t+\alpha}{[n+1]_{p,q}+\beta}b_n\right)-f(x)\right|\mathop{d_{p,q}t} \\
&& \hspace{1 cm} \leq M \sum_{k=0}^{n+m} \left[\begin{array}{c} n+m \\ k \end{array} \right]_{p,q}  \prod_{s=0}^{n+m-k-1} (p^s-q^s\dfrac{x}{b_n})\left(\dfrac{x}{b_n}\right)^k \\&& \hspace{2 cm} \times \int_0^1 \left| \dfrac{(1-t)[k]_{p,q}+[k+1]_{p,q}t+\alpha}{[n+1]_{p,q}+\beta}b_n-x\right|^\gamma\mathop{d_{p,q}t}.
\end{eqnarray*}
Applying H\"{o}lder's inequality with the values $p=\frac{2}{\gamma}$ and $q=\frac{2}{2-\gamma}$, we get following inequality,
\begin{eqnarray*}
&& \int_0^1 \left| \dfrac{(1-t)[k]_{p,q}+[k+1]_{p,q}t+\alpha}{[n+1]_{p,q}+\beta}b_n-x\right|^\gamma\mathop{d_{p,q}t} \\
&&  \qquad \leq \left\{ \int_0^1 \left( \dfrac{(1-t)[k]_{p,q}+[k+1]_{p,q}t+\alpha}{[n+1]_{p,q}+\beta}b_n-x\right)^2 \mathop{d_{p,q}t} \right\}^{\frac{\gamma}{2}}\left\{\int_0^1\mathop{d_{p,q}t}\right\}^{\frac{2-\gamma}{2}}
\\&&\qquad =  \left\{ \int_0^1 \left( \dfrac{(1-t)[k]_{p,q}+[k+1]_{p,q}t+\alpha}{[n+1]_{p,q}+\beta}b_n-x\right)^2 \mathop{d_{p,q}t} \right\}^{\frac{\gamma}{2}}
\end{eqnarray*}
Using this, we get
\begin{eqnarray*}
&&| K_{n,m}^{\alpha,\beta}(f;x,p,q)- f(x)|\\
&& \qquad \leq  M \sum_{k=0}^{n+m} \left[\begin{array}{c} n+m \\ k \end{array} \right]_{p,q}  \prod_{s=0}^{n+m-k-1} (p^s-q^s\dfrac{x}{b_n})\left(\dfrac{x}{b_n}\right)^k \\&& \hspace{2 cm} \times \left\{ \int_0^1 \left( \dfrac{(1-t)[k]_{p,q}+[k+1]_{p,q}t+\alpha}{[n+1]_{p,q}+\beta}b_n-x\right)^2 \mathop{d_{p,q}t} \right\}^{\frac{\gamma}{2}}
 \\&& \qquad = M \sum_{k=0}^{n+m} \left\{ \int_0^1 \left( \dfrac{(1-t)[k]_{p,q}+[k+1]_{p,q}t+\alpha}{[n+1]_{p,q}+\beta}b_n-x\right)^2 \mathop{d_{p,q}t} \right\}^{\frac{\gamma}{2}} w_{n,k}(p,q;x), 
\end{eqnarray*}
where $ w_{n,k}(p,q;x)=\big[\begin{array}{c} n+m \\ k \end{array} \big]_{p,q}  \prod_{s=0}^{n+m-k-1} (p^s-q^s\frac{x}{b_n})\left(\frac{x}{b_n}\right)^k $. Again using 
H\"{o}lder's inequality with $p=\frac{2}{\gamma}$ and $q=\frac{2}{2-\gamma}$, we have
\begin{eqnarray*}
&&| K_{n,m}^{\alpha,\beta}(f;x,p,q)- f(x)|\\
&& \qquad \leq M \left\{ \sum_{k=0}^{n+m} \int_0^1 \left( \dfrac{(1-t)[k]_{p,q}+[k+1]_{p,q}t+\alpha}{[n+1]_{p,q}+\beta}b_n-x\right)^2 \mathop{d_{p,q}t}~ w_{n,k}(p,q;x) \right\}^{\frac{\gamma}{2}} \left\{\sum_{k=0}^{n+m}w_{n,k}(p,q;x)  \right\}^{\frac{2-\gamma}{2}}
\\&& \qquad = M \left\{ \sum_{k=0}^{n+m} w_{n,k}(p,q;x) \int_0^1 \left( \dfrac{(1-t)[k]_{p,q}+[k+1]_{p,q}t+\alpha}{[n+1]_{p,q}+\beta}b_n-x\right)^2 \mathop{d_{p,q}t} \right\}^{\frac{\gamma}{2}}
\\&& \qquad = M (\mu_{n,p,q}(x))^{\gamma /2},
\end{eqnarray*}
where $(\mu_{n,p,q}(x))^{\gamma /2}=K_{n,m}^{\alpha,\beta}((t-x)^2;x,p,q)$.
\end{proof} 
In order to obtain rate of convergence in terms of modulus of continuity $\omega(f;\delta)$, we assume that for any uniformly continuous $f \in C_B[0,\infty)$ and $x \geq 0$, modulus of continuity of $f$ is given by
\begin{equation}
\omega(f;\delta)=\max_{\substack
{|t-x|\leq \delta \\ t,x\in [0,\infty)}}|f(t)-f(x)|.
\end{equation}
Thus it implies for any $\delta > 0$
\begin{equation}\label{E2}
|f(x)-f(y)| \leq \omega(f;\delta)\left( \dfrac{|x-y|}{\delta}+1\right),
\end{equation}
is satisfied.
\begin{theorem}\label{T2} If $f \in C_B[0,\infty),$ we have 
$$|K_{n,m}^{\alpha,\beta}(f;x,p,q)- f(x)| \leq 2 \omega (f;\sqrt[]{\mu_{n,p,q}(x)}),$$
where $\omega(f;\cdot)$ is modulus of continuity of $f$ and $\lambda_{n,p,q}(x)$ be the same
as in Theorem \ref{T1}.
\end{theorem}
\begin{proof}
 Using triangular inequality, we get
\begin{eqnarray*}
&&|K_{n,m}^{\alpha,\beta}(f;x,p,q)- f(x)|\\
&&\qquad = \Biggl|\sum_{k=0}^{n+m} \left[\begin{array}{c} n+m \\ k \end{array} \right]_{p,q} \left(\dfrac{x}{b_n}\right)^k \prod_{s=0}^{n+m-k-1} (p^s-q^s\dfrac{x}{b_n}) \left(f \left(\dfrac{(1-t)[k]_{p,q}+[k+1]_{p,q}t+\alpha}{[n+1]_{p,q}+\beta}b_n\right)-f(x)\right)\Biggl|\\
&& \qquad \leq   \sum_{k=0}^{n+m} \left[\begin{array}{c} n+m \\ k \end{array} \right]_{p,q} \left(\dfrac{x}{b_n}\right)^k \prod_{s=0}^{n+m-k-1} (p^s-q^s\dfrac{x}{b_n}) \left|f \left(\dfrac{(1-t)[k]_{p,q}+[k+1]_{p,q}t+\alpha}{[n+1]_{p,q}+\beta}b_n\right)-f(x)\right|,
\end{eqnarray*}
Now using (\ref{E2}) and H\"{o}lder's inequality, we get
\begin{eqnarray*}
&&|K_{n,m}^{\alpha,\beta}(f;x,p,q)- f(x)|\\
&&\qquad = \sum_{k=0}^{n+m} \left[\begin{array}{c} n+m \\ k \end{array} \right]_{p,q} \left(\dfrac{x}{b_n}\right)^k \prod_{s=0}^{n+m-k-1} (p^s-q^s\dfrac{x}{b_n}) \left(\dfrac{|\frac{(1-t)[k]_{p,q}+[k+1]_{p,q}t+\alpha}{[n+1]_{p,q}+\beta}b_n-x|}{\delta} +1\right)\omega(f;\delta)\\
&&\qquad \leq  \omega(f;\delta) \sum_{k=0}^{n+m} \left[\begin{array}{c} n+m \\ k \end{array} \right]_{p,q} \left(\dfrac{x}{b_n}\right)^k \prod_{s=0}^{n+m-k-1} (p^s-q^s\dfrac{x}{b_n})
\\&&\qquad \qquad +\dfrac{\omega(f;\delta)}{\delta} \sum_{k=0}^{n+m} \left[\begin{array}{c} n+m \\ k \end{array} \right]_{p,q} \left(\dfrac{x}{b_n}\right)^k \prod_{s=0}^{n+m-k-1} (p^s-q^s\dfrac{x}{b_n}) \left|\frac{(1-t)[k]_{p,q}+[k+1]_{p,q}t+\alpha}{[n+1]_{p,q}+\beta}b_n-x\right| \\
&& \qquad = \omega(f;\delta)+\dfrac{\omega(f;\delta)}{\delta}\Biggl\{\sum_{k=0}^{n+m} \left[\begin{array}{c} n+m \\ k \end{array} \right]_{p,q} \left(\dfrac{x}{b_n}\right)^k \prod_{s=0}^{n+m-k-1} (p^s-q^s\dfrac{x}{b_n}) \\&& \qquad \qquad \times \left(\frac{(1-t)[k]_{p,q}+[k+1]_{p,q}t+\alpha}{[n+1]_{p,q}+\beta}b_n-x\right)^2\Biggl\}^{\frac{1}{2}}\\
&&\qquad = \omega(f;\delta)+\dfrac{\omega(f;\delta)}{\delta}\left\{K_{n,m}^{\alpha,\beta}((t-x)^2;x,p,q)\right\}^{1/2}.
\end{eqnarray*}
Now choosing $\delta=\mu_{n,p,q}(x)$ as in Theorem \ref{T1}, we have
$$|K_{n,m}^{\alpha,\beta}(f;x,p,q)- f(x)| \leq 2 \omega(f;\sqrt{\mu_{n,p,q}(x)}). $$
\end{proof}
Now let us denote by $C_B^2[0,\infty)$  the space of all functions $f \in C_B[0,\infty)$ such that $f', f''\in C_B[0,\infty)$. Let $\|f \|$ denote the usual supremum norm of $f$ . The classical Peetre's
K-functional and the second modulus of smoothness of the function $f \in C_B[0,\infty)$ are defined
respectively as
\begin{equation*}
K(f,\delta):= \inf_{g\in C_B^2[0,\infty)}[\|f-g\|+\delta \|g''\|]
\end{equation*}
and
\begin{equation*}
\omega_2(f,\delta)=\sup_{\substack
{0<h<\delta, \\ x,x+h \in I}}|f(x+2h)-2f(x+h)+f(x)|,
\end{equation*}
where $\delta>0$. It is known that [see \cite{devore}, p. 177] there exists a constant $A>0$ such that
\begin{equation}
K(f,\delta)\leq A\omega_2(f,\delta).
\end{equation}
\begin{theorem}\label{T3}
Let $x \in [0,b_n]$ and $f\in C_B[0,\infty)$. Then, for fixed $p \in \mathbb{N}_0$, we have 
\begin{equation*}
|K_{n,m}^{\alpha,\beta}(f;x,p,q)-f(x)|\leq C\omega_2 (f,\sqrt{\alpha_{n,p,q}(x)})+\omega (f,\beta_{n,p,q}(x))
\end{equation*}
for some positive constant $C$, where
\begin{eqnarray}\label{alpha}
&&\alpha_{n,p,q}(x)=\Biggl[\Biggl\{ 1+\dfrac{2(q-1)}{[2]_{p,q}}+\dfrac{(q-1)^2}{[3]_{p,q}}+\dfrac{(p+2q-1)^2}{[2]_{p,q}^2} \Biggl\}\dfrac{[n+m]_{p,q}^2}{([n+1]_{p,q}+\beta)^2}-4 \dfrac{(p+2q-1)[n+m]_{p,q}}{[2]_{p,q}([n+1]_{p,q}+\beta)}+2 \Biggl] x^2 
\nonumber \\&& \hspace{1 cm}+ \Biggl[
\left\{ 1+\dfrac{2q}{[2]_{p,q}}+\dfrac{q^2-1}{[3]_{p,q}}+2\dfrac{(p+2q-1)}{[2]_{p,q}^2}\right\} \dfrac{[n+m]_{p,q}}{([n+1]_{p,q}+\beta)^2}\left(p\frac{x}{b_n}+1-\frac{x}{b_n}\right)_{p,q}^{n+m}\nonumber \\ && \hspace{3 cm}
+4\dfrac{\alpha (p+2q-1)[n+m]_{p,q}}{[2]_{p,q}([n+1]_{p,q}+\beta)^2}-4\dfrac{(p\frac{x}{b_n}+1-\frac{x}{b_n})_{p,q}^{n+m}}{[2]_{p,q}([n+1]_{p,q}+\beta)}-4\dfrac{\alpha}{([n+1]_{p,q}+\beta)}
\Biggl]b_nx \nonumber \\&& \hspace{1 cm}
+ \Biggl[ \dfrac{(p^2\frac{x}{b_n}+1-\frac{x}{b_n})_{p,q}^{n+m}}{[3]_{p,q}}+\dfrac{(p\frac{x}{b_n}+1-\frac{x}{b_n})_{p,q}^{2n+2m}}{[2]_{p,q}^2}+4\dfrac{\alpha}{[2]_{p,q}}(p\frac{x}{b_n}+1-\frac{x}{b_n})_{p,q}^{n+m}+2\alpha^2\Biggl]\dfrac{b_n^2}{([n+1]_{p,q}+\beta)^2} ,
\end{eqnarray}
and
\begin{equation}\label{beta}
\beta_{n,p,q}(x)=\Biggl(\dfrac{[2]_{p,q}\alpha+(p\frac{x}{b_n}+1-\frac{x}{b_n})_{p,q}^{n+m}}{[2]_{p,q}([n+1]_{p,q}+\beta)}b_n+\Biggl(\dfrac{(p+2q-1)[n+m]_{p,q}}{[2]_{p,q}([n+1]_{p,q}+\beta)}-1\Biggl)x \Biggl).
\end{equation}
\end{theorem}
\begin{proof}
Consider an auxiliary operator $K_{n,m}^{*}(f;x,p,q): C_B[0,\infty) \to C_B[0,\infty)$ by
\begin{equation}\label{operator}
K_{n,m}^{*}(f;x,p,q):=K_{n,m}^{\alpha,\beta}(f;x,p,q)-f\left(\dfrac{[2]_{p,q}\alpha+(p\frac{x}{b_n}+1-\frac{x}{b_n})_{p,q}^{n+m}}{[2]_{p,q}([n+1]_{p,q}+\beta)}b_n+\dfrac{(p+2q-1)[n+m]_{p,q}}{[2]_{p,q}([n+1]_{p,q}+\beta)}x\right)+f(x).
\end{equation}
Then by Lemma \ref{L2} we get
\begin{align}
\label{eqn:eqlabel}
\begin{split}
K_{n,m}^{*}(1;x,p,q)&=1,
\\
K_{n,m}^{*}((t-x);x,p,q)&=0.
\end{split}
\end{align}
For given $g \in C_B [0,\infty)$, it follows by the Taylor formula that
\begin{equation*}
g(y) - g(x) = (y - x)g'(x) + \int_x^y (y - u)g''(u) \mathop{du}.
\end{equation*}
Taking into account \ref{operator} and using \ref{eqn:eqlabel}, we get
\begin{eqnarray*}
|K_{n,m}^{*}(g;x,p,q)-g(x)|&=&|K_{n,m}^{*}(g(y)-g(x);x,p,q)|\\
&=& \left|g'(x)K_{n,m}^{*}((y-x);x,p,q)+K_{n,m}^{*}\left(\int_x^y(y-u)g''(u) \mathop{du};x,p,q\right)\right|\\
&=& \left|K_{n,m}^{*}\left(\int_x^y(y-u)g''(u) \mathop{du};x,p,q\right)\right|
\end{eqnarray*}
Then by \ref{operator}
\begin{eqnarray*}
&&|K_{n,m}^{*}(g;x,p,q)-g(x)| \\ &&
=\Biggl|K_{n,m}^{*}\left(\int_x^y(y-u)g''(u) \mathop{du};x,p,q\right)\\
&&\hspace{.1 cm}
- \int_x^{\footnotesize{\left(\frac{[2]_{p,q}\alpha+(p\frac{x}{b_n}+1-\frac{x}{b_n})_{p,q}^{n+m}}{[2]_{p,q}([n+1]_{p,q}+\beta)}b_n+\frac{(p+2q-1)[n+m]_{p,q}}{[2]_{p,q}([n+1]_{p,q}+\beta)}x\right)}}
 \Biggl(
\dfrac{[2]_{p,q}\alpha+(p\frac{x}{b_n}+1-\frac{x}{b_n})_{p,q}^{n+m}}{[2]_{p,q}([n+1]_{p,q}+\beta)}b_n+\dfrac{(p+2q-1)[n+m]_{p,q}}{[2]_{p,q}([n+1]_{p,q}+\beta)}x
-u\Biggl)g''(u)\mathop{du}\Biggl|
\\&& \hspace{.1 cm}
\leq \Biggl|K_{n,m}^{*}\left(\int_x^y(y-u)g''(u) \mathop{du};x,p,q\right)\Biggl|+\Biggl|\int_x^{\footnotesize{\left(\frac{[2]_{p,q}\alpha+(p\frac{x}{b_n}+1-\frac{x}{b_n})_{p,q}^{n+m}}{[2]_{p,q}([n+1]_{p,q}+\beta)}b_n+\frac{(p+2q-1)[n+m]_{p,q}}{[2]_{p,q}([n+1]_{p,q}+\beta)}x\right)}}\\&& \hspace{2 cm}\times
 \Biggl(
\frac{[2]_{p,q}\alpha+(p\frac{x}{b_n}+1-\frac{x}{b_n})_{p,q}^{n+m}}{[2]_{p,q}([n+1]_{p,q}+\beta)}b_n+\frac{(p+2q-1)[n+m]_{p,q}}{[2]_{p,q}([n+1]_{p,q}+\beta)}x
-u\Biggl)g''(u)\mathop{du} \Biggl|.
\end{eqnarray*}
Since,
\begin{equation*}
\Bigl|K_{n,m}^{\alpha,\beta}\left(\int_x^y(y-u)g''(u) \mathop{du};x,p,q\right)\Bigl| \leq \|g''(x)\|~K_{n,m}^{\alpha,\beta}((y-x)^2;x,p,q)
\end{equation*}
and
\begin{eqnarray*}
&&\Biggl|\int_x^{\footnotesize{\left(\frac{[2]_{p,q}\alpha+(p\frac{x}{b_n}+1-\frac{x}{b_n})_{p,q}^{n+m}}{[2]_{p,q}([n+1]_{p,q}+\beta)}b_n+\frac{(p+2q-1)[n+m]_{p,q}}{[2]_{p,q}([n+1]_{p,q}+\beta)}x\right)}}
  \Biggl(
 \frac{[2]_{p,q}\alpha+(p\frac{x}{b_n}+1-\frac{x}{b_n})_{p,q}^{n+m}}{[2]_{p,q}([n+1]_{p,q}+\beta)}b_n+\frac{(p+2q-1)[n+m]_{p,q}}{[2]_{p,q}([n+1]_{p,q}+\beta)}x
-u\Biggl)g''(u)\mathop{du} \Biggl|\\&&\hspace{1 cm}\leq ~\|g''\| \Biggl[\dfrac{[2]_{p,q}\alpha+(p\frac{x}{b_n}+1-\frac{x}{b_n})_{p,q}^{n+m}}{[2]_{p,q}([n+1]_{p,q}+\beta)}b_n+\Biggl(\dfrac{(p+2q-1)[n+m]_{p,q}}{[2]_{p,q}([n+1]_{p,q}+\beta)}-1\Biggl)x  \Biggl]^2,
\end{eqnarray*}
we get
\begin{eqnarray*}
&&|K_{n,m}^{*}(g;x,p,q)-g(x)| \leq \|g''\|K_{n,m}^{\alpha,\beta}((y-x)^2;x,p,q)+\|g''\|\Biggl[\dfrac{[2]_{p,q}\alpha+(p\frac{x}{b_n}+1-\frac{x}{b_n})_{p,q}^{n+m}}{[2]_{p,q}([n+1]_{p,q}+\beta)}b_n\\&& \hspace{8 cm}+\Biggl(\dfrac{(p+2q-1)[n+m]_{p,q}}{[2]_{p,q}([n+1]_{p,q}+\beta)}-1\Biggl)x\Biggl]^2.
\end{eqnarray*}
Hence Lemma \ref{L2} implies that
\begin{eqnarray}\label{E4}
&&|K_{n,m}^{*}(g;x,p,q)-g(x)| \leq \|g''\|\Biggl[
\left(\dfrac{\alpha^2}{([n+1]_{p,q}+\beta)^2}+\dfrac{2\alpha}{[2]_{p,q}([n+1]_{p,q}+\beta)^2}\left(p\frac{x}{b_n}+1-\frac{x}{b_n}\right)_{p,q}^{n+m}+\dfrac{(p^2\frac{x}{b_n}+1-\frac{x}{b_n})_{p,q}^{n+m}}{[3]_{p,q}([n+1]_{p,q}+\beta)^2}\right)b_n^2 \nonumber\\&&\hspace{4.5 cm}
+\Biggl(\dfrac{2\alpha(p+2q-1)[n+m]_{p,q}}{[2]_{p,q}([n+1]_{p,q}+\beta)^2}+
\left\{ 1+\dfrac{2q}{[2]_{p,q}}+\dfrac{q^2-1}{[3]_{p,q}}\right\} \dfrac{[n+m]_{p,q}}{([n+1]_{p,q}+\beta)^2} \left(p\frac{x}{b_n}+1-\frac{x}{b_n}\right)_{p,q}^{n+m-1}\nonumber
\\&& \hspace{5.5 cm}-\dfrac{2\alpha}{([n+1]_{p,q}+\beta)} - \dfrac{2(p\frac{x}{b_n}+1-\frac{x}{b_n})_{p,q}^{n+m}}{[2]_{p,q}([n+1]_{p,q}+\beta)}
\Biggl)b_nx \nonumber \\&& \hspace{4.5 cm}+\Biggl(\Biggl\{ 1+\dfrac{2(q-1)}{[2]_{p,q}}+\dfrac{(q-1)^2}{[3]_{p,q}}\Biggl\}\dfrac{[n+m]_{p,q}[n+m-1]_{p,q}}{([n+1]_{p,q}+\beta)^2}-2 \dfrac{(p+2q-1)[n+m]_{p,q}}{[2]_{p,q}([n+1]_{p,q}+\beta)} +1 \Biggl)x^2
 \nonumber \\&& \hspace{4.5 cm}
+ \Biggl(\dfrac{[2]_{p,q}\alpha+(p\frac{x}{b_n}+1-\frac{x}{b_n})_{p,q}^{n+m}}{[2]_{p,q}([n+1]_{p,q}+\beta)}b_n+\Biggl(\dfrac{(p+2q-1)[n+m]_{p,q}}{[2]_{p,q}([n+1]_{p,q}+\beta)}-1\Biggl)x \Biggl)^2
 \Biggl].
\end{eqnarray}
Since $K_{n,m}^{*}(f;x,p,q) \leq 3\|f\|$, considering \ref{alpha} and \ref{beta},for all $f \in C_B[0,\infty)$ and $g \in C_B^2[0,\infty)$, we may write from \ref{E4} that
\begin{eqnarray*}
&&|K_{n,m}^{\alpha,\beta}(f;x,p,q)-f(x)|\leq |K_{n,m}^{*}(f-g;x,p,q)-(f-g)(x)|+|K_{n,m}^{*}(g;x,p,q)-g(x)|
\\&& \hspace{4 cm}+\Biggl|f\Biggl(\frac{[2]_{p,q}\alpha+(p\frac{x}{b_n}+1-\frac{x}{b_n})_{p,q}^{n+m}}{[2]_{p,q}([n+1]_{p,q}+\beta)}b_n+\frac{(p+2q-1)[n+m]_{p,q}}{[2]_{p,q}([n+1]_{p,q}+\beta)}x\Biggl)-f(x)\Biggl|
\\&& \hspace{4 cm}
\leq 4\|f-g\| + \alpha_{n,p,q}(x)\|g'\|+ \Biggl|f\Biggl(\frac{[2]_{p,q}\alpha+(p\frac{x}{b_n}+1-\frac{x}{b_n})_{p,q}^{n+m}}{[2]_{p,q}([n+1]_{p,q}+\beta)}b_n+\frac{(p+2q-1)[n+m]_{p,q}}{[2]_{p,q}([n+1]_{p,q}+\beta)}x\Biggl)-f(x)\Biggl|
\\&& \hspace{4 cm}
\leq 4\|f-g\| + \alpha_{n,p,q}(x)\|g'\|+ \omega(f,\beta_{n,p,q}(x)) ,
\end{eqnarray*}
which yields that
\begin{eqnarray*}
|K_{n,m}^{\alpha,\beta}(f;x,p,q)-f(x)|&\leq &4 K(f,\alpha_{n,p,q}(x))+\omega(f,beta_{n,p,q}(x))
\\&\leq& C\omega_2(f,\sqrt{\alpha_{n,p,q}(x)})+\omega(f,\beta_{n,p,q}(x)),
\end{eqnarray*}
where
\begin{eqnarray*}
&& \alpha_{n,p,q}(x)= \Biggl[\Biggl\{ 1+\dfrac{2(q-1)}{[2]_{p,q}}+\dfrac{(q-1)^2}{[3]_{p,q}}+\dfrac{(p+2q-1)^2}{[2]_{p,q}^2} \Biggl\}\dfrac{[n+m]_{p,q}^2}{([n+1]_{p,q}+\beta)^2}-4 \dfrac{(p+2q-1)[n+m]_{p,q}}{[2]_{p,q}([n+1]_{p,q}+\beta)}+2 \Biggl] x^2 
\\&& \hspace{2cm}+ \Biggl[
\left\{ 1+\dfrac{2q}{[2]_{p,q}}+\dfrac{q^2-1}{[3]_{p,q}}+2\dfrac{(p+2q-1)}{[2]_{p,q}^2}\right\} \dfrac{[n+m]_{p,q}}{([n+1]_{p,q}+\beta)^2}\left(p\frac{x}{b_n}+1-\frac{x}{b_n}\right)_{p,q}^{n+m}\\ && \hspace{3 cm}
+4\dfrac{\alpha (p+2q-1)[n+m]_{p,q}}{[2]_{p,q}([n+1]_{p,q}+\beta)^2}-4\dfrac{(p\frac{x}{b_n}+1-\frac{x}{b_n})_{p,q}^{n+m}}{[2]_{p,q}([n+1]_{p,q}+\beta)}-4\dfrac{\alpha}{([n+1]_{p,q}+\beta)}
\Biggl]b_nx \\&& \hspace{2 cm}
+ \Biggl[ \dfrac{(p^2\frac{x}{b_n}+1-\frac{x}{b_n})_{p,q}^{n+m}}{[3]_{p,q}}+\dfrac{(p\frac{x}{b_n}+1-\frac{x}{b_n})_{p,q}^{2n+2m}}{[2]_{p,q}^2}+4\dfrac{\alpha}{[2]_{p,q}}(p\frac{x}{b_n}+1-\frac{x}{b_n})_{p,q}^{n+m}+2\alpha^2\Biggl]\dfrac{b_n^2}{([n+1]_{p,q}+\beta)^2},
\end{eqnarray*}
and
\begin{equation*}
\beta_{n,p,q}(x)=\Biggl(\dfrac{[2]_{p,q}\alpha+(p\frac{x}{b_n}+1-\frac{x}{b_n})_{p,q}^{n+m}}{[2]_{p,q}([n+1]_{p,q}+\beta)}b_n+\Biggl(\dfrac{(p+2q-1)[n+m]_{p,q}}{[2]_{p,q}([n+1]_{p,q}+\beta)}-1\Biggl)x \Biggl).
\end{equation*}
Hence we get the result.
\end{proof}
\begin{flushleft}
\textbf{Conflict of Interest} The authors declare that there is no conflict of interests.
\end{flushleft}
\begin{flushleft}
\textbf{References}
\end{flushleft}


\begin{thebibliography}{99}

\bibitem{acar} T. Acar, $(p, q)$-Generalization of Sz\'{a}sz-Mirakyan operators, \textit{Math. Meth. Appl. Sci., 2015}, DOI: 10.1002/mma.3721.

\bibitem{pqdiff} I. M. Burban, A. U. Klimyk, $(p, q)$-differentiation, $(p, q)$-integration, and $(p, q)$-hypergeometric functions related to quantum groups, \textit{Integral Transforms and Special Functions}, 1994, Vol. 2, No. 1, pp.15-36.

\bibitem{HM4} N. L. Braha, H. M. Srivastava and S. A. Mohiuddine, A Korovkin's type approximation
theorem for periodic functions via the statistical summability of the generalized de la Vall\'{e}e
Poussin mean, Appl. Math. Comput. 228 (2014) 162-169.

\bibitem{devore} R.A. Devore, G.G. Lorentz, Constructive Approximation, \textit{Springer, Berlin},1993.

\bibitem{korovUnbounded} A. D. Gadjiev, The convergence problem for a sequence of positive linear operators on unbounded sets and theorems analogues to that of P.P. Korovkin, \textit{Dokl. Akad. Nauk SSSR}, 218 (5), 1001-1004. English translation in Sov. Math. Dokl., 15 (5), 1974,  pp. 1433-1436.

\bibitem{Gairola} A.R. Gairola, Deepmala, L.N. Mishra, Rate of Approximation by Finite Iterates of q-Durrmeyer Operators, Proceedings of the National Academy of Sciences, India Section A: Physical Sciences, (2016), doi: 10.1007/s40010-016-0267-z, in press.

\bibitem{Ibikli} E. Ibikli, Approximation by Bernstein-Chlodowsky polynomials, \textit{Hacettepe Journal of Mathematics and Statistics,} Vol. 32 (2003), pp. 1-5.

\bibitem{mishra1} V.N. Mishra, K. Khatri, L.N. Mishra, Deepmala; Inverse result in simultaneous approximation by Baskakov-Durrmeyer-Stancu operators, Journal of Inequalities and Applications 2013, 2013:586. doi:10.1186/1029-242X-2013-586.

\bibitem{mishra2} V.N. Mishra, K. Khatri, L.N. Mishra; On Simultaneous Approximation for Baskakov-Durrmeyer-Stancu type operators, Journal of Ultra Scientist of Physical Sciences, Vol. 24, No. (3) A, 2012, pp. 567-577.

\bibitem{mishra3} V.N. Mishra, K. Khatri, L.N. Mishra; Statistical approximation by Kantorovich type Discrete $q-$Beta operators, Advances in Difference Equations 2013, 2013:345, DOI: 10.1186/10.1186/1687-1847-2013-345.

\bibitem{mishra4} V.N. Mishra, S. Pandey, On $(p, q)$ Baskakov-Durrmeyer-Stancu Operators, arXiv:1602.06719 [math.CA]

\bibitem{mishra5} V.N. Mishra, S. Pandey, $(p, q)$-Sz\'asz-Mirakyan-Baskakov-Stancu type Operators, arXiv:1602.06312 [math.CA].

\bibitem{mishra6} V.N. Mishra, S. Pandey, $(p, q)$-Sz\'{a}sz-Mirakyan-Stancu-Chlodowsky type Operators, communicated.

\bibitem{bernpq} M. Mursaleen, Khursheed J. Ansari, A. Khan, On $(p,q)$-analogue of Bernstein operators, \textit{Appl. Math. Comput.,} 266 (2015), pp. 874-882.

\bibitem{bernpqstancu} M. Mursaleen, Khursheed J. Ansari, A. Khan, Some approximation results by $(p,q)$-analogue of Bernstein-Stancu operators, \textit{Appl. Math. Comput.,} 264 (2015), pp. 392-402.

\bibitem{bernpqkonto} M. Mursaleen, F. Khan, Approximation by Kantorovich type $(p,q)$-Bernstein Schurer operators, arXiv:1506.02492 [math.CA].

\bibitem{blehahn} M. Mursaleen, Md. Nasiruzzaman, Asif Khan, Khursheed J. Ansari, Some approximation results on Bleimann-Butzer-Hahn operators defined by $(p, q)$-integers, arXiv:1505.00392, (2015).

\bibitem{HM1} M. Mursaleen, A. Khan, H. M. Srivastava and K. S. Nisar, Operators constructed by means of $q$-Lagrange polynomials and A-statistical approximation, Appl. Math. Comput. 219 (2013), 6911-6818.

\bibitem{HM2} H. M. Srivastava, Some generalizations and basic (or $q$-) extensions
of the Bernoulli, Euler and Genocchi polynomials, Appl. Math. Inform. Sci. 5 (2011), 390-444.

\bibitem{HM3} H. M. Srivastava and J. Choi, Zeta and $q$-Zeta Functions and
Associated Series and Integrals, Elsevier Science Publishers, Amsterdam,
London and New York, 2012.

\bibitem{bskq} T. Vedi and Mehmet Ali \"{O}zarslan, Chlodowsky-type $q$-Bernstein-Stancu-Kantorovich operators, \textit{J. Inequal. Appl.,} 91(2015).

\bibitem{BaskontoStancu} A. Wafi, N. Rao, Deepmala, Approximation properties by generalized-Baskakov-Kantorovich-Stancu type operators, Appl. Math. Inf. Sci. Lett., Vol.  4, No. 3, (2016), pp. 1-8.


\end{thebibliography}
\end{document}